\documentclass{article}
\usepackage{amssymb,amscd,amsthm, verbatim,amsmath,color,fancyhdr, mathrsfs}
\usepackage{tikz}
\usepackage{graphicx}
\usepackage{turnstile}
\usepackage{mathtools}
\usepackage{bm}
\usepackage{inconsolata}  
\usepackage{comment}

\usepackage[colorlinks=true,citecolor=blue]{hyperref}

\usepackage[all]{xy}
\usepackage{float}
\usepackage{bbold}
\usepackage{enumerate}
\usepackage{listings}

\newcommand{\RR}{\mathbb{R}}

\newcommand{\NN}{\mathcal{N}}

\newcommand{\FF}{\mathcal{F}}

\newcommand{\MM}{\mathcal{M}}

\newcommand{\lk}{\mathsf{lk}}

\newcommand{\CC}{\mathsf{C}}

\newcommand{\HH}{\mathcal{H}}

\newcommand{\sK}{\mathsf{K}}
\newcommand{\sL}{\mathsf{L}}
\newcommand{\sB}{\mathsf{B}}

\newcommand{\cS}{\mathsf{S}}
\newcommand{\XX}{\mathsf{X}}
\newcommand{\YY}{\mathsf{Y}}

\newtheorem{thm}{Theorem}[section]
\newtheorem{Q}{Question}

\newtheorem{lemma}[thm]{Lemma}
\newtheorem*{claim}{Claim}
\newtheorem{prop}[thm]{Proposition}

\theoremstyle{definition}
\newtheorem{Def}[thm]{Definition}
\newtheorem{remark}[thm]{Remark}
\newtheorem{eg}[thm]{Example}

\DeclareMathOperator{\conv}{conv}

\usepackage[normalem]{ulem} 
\usepackage{bm}

\usepackage[mathlines]{lineno}

\newcommand*\patchAmsMathEnvironmentForLineno[1]{%
  \expandafter\let\csname old#1\expandafter\endcsname\csname #1\endcsname
  \expandafter\let\csname oldend#1\expandafter\endcsname\csname end#1\endcsname
  \renewenvironment{#1}%
     {\linenomath\csname old#1\endcsname}%
     {\csname oldend#1\endcsname\endlinenomath}}%
\newcommand*\patchBothAmsMathEnvironmentsForLineno[1]{%
  \patchAmsMathEnvironmentForLineno{#1}%
  \patchAmsMathEnvironmentForLineno{#1*}}%
\AtBeginDocument{%
\patchBothAmsMathEnvironmentsForLineno{equation}%
\patchBothAmsMathEnvironmentsForLineno{align}%
\patchBothAmsMathEnvironmentsForLineno{flalign}%
\patchBothAmsMathEnvironmentsForLineno{alignat}%
\patchBothAmsMathEnvironmentsForLineno{gather}%
\patchBothAmsMathEnvironmentsForLineno{multline}%
}

\title{Transversals and colorings of simplicial spheres}

\date{} 
\usepackage[utf8]{inputenc}
\usepackage{authblk} 

\author[1]{Joseph Briggs\thanks{jgb0059@auburn.edu}}
\author[2]{Michael Gene Dobbins\thanks{mdobbins@binghamton.edu}}
\author[3]{Seunghun Lee\thanks{seunghun.lee@mail.huji.ac.il}\thanks{Corresponding author}}

\affil[1]{Department of Mathematics and Statistics, Auburn University, Auburn, AL, USA.}
\affil[2]{Department of Mathematical Sciences, Binghamton University, Binghamton, NY, USA. }
\affil[3]{Einstein Institute of Mathematics, Hebrew University, Jerusalem, Israel.}

\begin{document}
\maketitle

\begin{abstract}
	Motivated from the surrounding property of a point set in $\RR^d$ introduced by Holmsen, Pach and Tverberg, we consider the transversal number and chromatic number of a simplicial sphere. As an attempt to give a lower bound for the maximum transversal ratio of simplicial $d$-spheres, we provide two infinite constructions. The first construction  gives infintely many $(d+1)$-dimensional simplicial polytopes with the transversal ratio exactly $\frac{2}{d+2}$ for every $d\geq 2$. In the case of $d=2$, this meets the previously well-known upper bound $1/2$ tightly. The second gives infinitely many simplicial 3-spheres with the transversal ratio greater than $1/2$.  This was unexpected from what was previously known about the surrounding property. Moreover, we show that, for $d\geq 3$, the facet hypergraph $\mathcal{F}(\sK)$ of a $d$-dimensional simplicial sphere $\sK$ has the chromatic number $\chi(\mathcal{F}(\sK)) \in O(n^{\frac{\lceil d/2\rceil-1}{d}})$, where $n$ is the number of vertices of $\sK$. This slightly improves the upper bound previously obtained by Heise, Panagiotou, Pikhurko, and Taraz.
\end{abstract}



\section{Introduction}\label{section_intro}

Given a hypergraph $\HH$ on $n$ vertices, a \textit{(weak) coloring} of $\HH$ is a function $c:V(\HH)\to [m]=\{1, \dots, m\}$ such that for every \textit{color} $i\in [m]$, $c^{-1}(i)$ is an \textit{independent set} of $\HH$, that is, $c^{-1}(i)$ does not contain any hyperedges of $\HH$. $\HH$ is said to be \textit{$m$-colorable} if there is a proper coloring of $\HH$ using $m$ colors. A subset $T$ of $V(\HH)$ is called a \textit{transversal} if for every hyperedge $e\in \HH$ there exists a vertex $v\in T$ such that $v \in e$, or equivalently, the complement $V(\HH)\setminus T$ is an independent set of $\HH$. The minimum number of colors required in a coloring of $\HH$ is called the \textit{(weak) chromatic number} of $\HH$ and we denote it by $\chi(\HH)$. The minimum size of a transversal of $\HH$ is called the \textit{transversal number} of $\HH$, and denoted by $\tau(\HH)$. The \textit{independence number} of $\HH$ is defined as the largest size of an independent set in $\HH$, and denoted by $\alpha(\HH)$.

\medskip 

For a simplicial complex $\sK$, we denote the \textit{facet hypergraph} of $\sK$, which consists of the facets of $\sK$, by $\FF(\sK)$. We similarly define the \textit{facet hypergraph of a polytope} $\mathbf{P}$ where each facet is represented as the set of its vertices, and denote it by $\FF(\mathbf{P})$. One of the goals of this note is to give a partial answer of the following question, which was originally proposed by Andreas Holmsen (personal communication).

\begin{Q}\label{question_trans}
Let $d$ be a fixed integer. How large can the transversal ratio $\tau(\FF(\sK))/n$ be as $n$ approaches to $\infty$, where $\sK$ is a $d$-dimensional simplicial sphere on $n$ vertices?
\end{Q}

Question \ref{question_trans} for simplicial polytopes was implicitly investigated in a dual form in \cite{holmsen_surrounding} by Holmsen, Pach and Tverberg. They considered the following properties which arise from combinatorial convexity.
Given a finite point set $P \subseteq \RR^d$ in general position with respect to the origin $O$, $P$ is said to be \textit{$k$-surrounding}, or is said to have \textit{property $S(k)$}, if any $k$-element subset of $P$ can be extended to a $(d+1)$-element subset of $P$ such that $O\in \conv(P)$. 
By applying the Gale transform to $P$, we can obtain a dual point set $P^*$ in $\RR^{n-d-1}$, and there is the dual property $S^*(k)$ of $P^*$ which is defined as the following: Among any $n-k$ points of $P^*$, there are some $n-d-1$ that form a facet of $\conv( P^*)$. Using fundamental facts of the Gale transform, one can easily see that $P$ has property $S(k)$ if and only if $P^*$ satisfies $S^*(k)$. Also, note that when points of $P^*$ are
in convex position, we have that $P^*$ satisfies $S^*(k)$ if and only if the facet hypergraph $\FF(\conv(P^*))$ has no transversal of size $k$.  

The $S(k)$ \textit{existence problem} \cite[Problem 3]{holmsen_surrounding}
asks for a point set in general position with respect to $O$ that satisfies $S(k)$ for large values of $k$ when the dimension $d$ is fixed. 
In \cite{holmsen_surrounding}, such a construction was obtained from even-dimensional cyclic polytopes $C_{d^*}(n)$ in the dual space which satisfy $S^*(\lfloor d/2 \rfloor+1)$, or equivalently $S^*(\lceil (n-d^*)/2 \rceil)$, where $d^*=n-d-1$ is the dimension of the dual space. Note that it is equivalent to $\tau(\FF(C_{d^*}(n)))>k$. In this way, their investigation of the dual surrounding property $S^*(k)$ has a connection to lower bounds on the transversal ratio of a simplicial sphere. 

\medskip 

The exact transversal number of a cyclic polytope is given by $\tau(\FF(C_{d^*}(n)))=\lceil (n-d^*)/2 \rceil+1$ when $d^*$ is even, and $\tau(\FF(C_{d^*}(n)))=2$ when $d^*$ is odd.
See Proposition \ref{prop_basic_cyclic_col} and a remark after the proof. This implies that, for $d^* \geq 3$, the transversal ratio $\tau/n$ of $C_{d^*}(n)$ is bounded by $1/2$ from above.  In addition to cyclic polytopes, it can be checked from the result in \cite[Section 4.8]{miyata} by Miyata and Padrol that the same bound works for all the neighborly oriented matroids constructed in that paper.  Furthermore, all examples except for odd cycles in Section \ref{section_basic_examples} satisfy 2-colorability which is an even stronger property.

One of the examples in Section \ref{section_basic_examples} is the class of 3-dimensional simplicial polytopes. For this case, we obtain tight lower bound constructions that have transversal ratio exactly $1/2$, which matches the upper bound for simplicial 3-polytopes.  See also Proposition \ref{prop_basic_3-polytope}. This construction generalizes to higher dimensions as follows.
\begin{thm}\label{thm_plane_tight}
For every $d \geq 2$, there are infinitely many $(d+1)$-dimensional simplicial polytopes with the transversal ratio exactly $\frac{2}{d+2}$. 
\end{thm}
Together with even-dimensional cyclic polytopes, the constructions of Theorem \ref{thm_plane_tight} give positive lower bounds to Question \ref{question_trans} for every dimension. 

\medskip

Even though these form a rather restricted set of examples, the bound of $1/2$ on the transversal ratio holds in surprisingly many cases, so it may be very tempting to conjecture that,
for a simplicial $d$-sphere $\sK$ and $d\geq 3$, $\FF(\sK)$ is 2-colorable, or at least satisfies the bound $\tau(\FF(\sK))\leq n/2+o(n)$. The following result demonstrates that this hope cannot be fulfilled.

\begin{thm}\label{thm_trans}
There exist arbitrarily large PL 3-spheres with transversal ratio at least $\frac{11}{21}$. 
\end{thm}

These PL spheres are constructed by bootstrapping the hypergraph $\NN_{21}$ which was obtained from computer experiments, see Theorem \ref{thm_construction}.

\medskip 

As we have already observed above, for any hypergraph $\HH$ with $n$ vertices, we have
\[
\tau(\HH)=n-\alpha(\HH) \leq n-n/\chi(\HH).
\]
So a lower bound on the transversal ratio $\tau(\HH)/n$ implies a lower bound on the chromatic number $\chi(\HH)$. This motivates the following analogue of Question \ref{question_trans} for chromatic numbers.

\begin{Q}\label{question_chi}
Under the same conditions of Question \ref{question_trans}, how large can $\chi(\FF(\sK))$ be?
\end{Q}	

Theorem \ref{thm_trans} implies that we need at least 3 colors to properly color the facet hypergraph of a simplicial 3-sphere in general. Unfortunately, we do not have an upper bound for the chromatic number which closes the gap with this finite lower bound. 

The most relevant result regarding upper bounds on the chromatic number of a simplicial sphere can be found in \cite{coloring_d-embeddable}. There, Heise, Panagiotou, Pikhurko, and Taraz obtained an upper bound on the chromatic number of linearly $d$-embeddable $(d+1)$-uniform hypergraphs. They used the Lov\'asz local lemma and the upper bound theorem by Dey and Pach \cite{upper_bound_pach} on the number of hyperedges as main ingredients. Following the same argument and using Stanley's upper bound theorem on simplicial spheres \cite{upper_bound_stanley} instead, one can easily get the same asymptotic upper bound $\chi(\FF(\sK))\in O(n^{\frac{\lceil d/2\rceil}{d}})$ for a simplicial $d$-sphere $\sK$.

Our new result is a slight improvement on this by a different probabilistic approach. Even though the result holds generally for simplicial spheres, we state and prove it for polytopes first for simplicity of arguments. Then, we explain how it can be generalized for simplicial spheres in Subsection \ref{subsection_upper_bound_embeddable}.

\begin{thm}\label{thm_main_upper_bound}
Let $\mathbf{P}$ be a $(d+1)$-dimensional simplicial polytope on $n$ vertices. Then, the facet hypergraph $\FF(\mathbf{P})$ has the chromatic number 
$$\chi(\FF(\mathbf{P}))\in O(n^{\frac{\lceil d/2\rceil-1}{d}}).$$
\end{thm}

\medskip

This paper is organized as follows. In Section \ref{section_preliminaries}, we introduce relevant notations and basic concepts.
In Section \ref{section_basic_examples}, we consider simplicial spheres of special classes and investigate their transversal numbers and colorability. In Section \ref{section_plane_tight}, we give the construction for Theorem \ref{thm_plane_tight}. In Section \ref{section_non-2-colorable}, we present the proof of Theorem \ref{thm_trans} which makes use of the construction in Appendix  \ref{section_appendix} obtained by computer experiments. In Section \ref{section_upper_bound}, we prove Theorem \ref{thm_main_upper_bound} and generalize it for simplicial spheres (see Theorem \ref{thm_main_upper_bound_simplicial_sphere}).
Finally, we give remarks on future research in Section \ref{section_final_remark}.

\section{Preliminaries}\label{section_preliminaries}
A \textit{hypergraph} $\HH$ on a vertex set $V=V(\HH)$, is a family of subsets of $V$, which are called \textit{hyperedges}. We assume that the hyperedges of $\HH$ cover the whole $V(\HH)$, that is, $\bigcup_{f\in \HH} f= V(\HH)$. A hypergraph which is a subfamily of $\HH$ is called a \textit{subhypergraph} of $\HH$. $\HH$ is called \textit{$k$-uniform} if all hyperedges have the same size $k$.  For two hypergraphs $\HH_1$ and $\HH_2$, we define the \textit{join} operation $*$ as
$$\HH_1*\HH_2=\{f_1\cup f_2: f_1\in \HH_1, f_2\in \HH_2\}.$$

A \textit{simplicial complex} $\sK$ on a vertex set $V$ is a hypergraph on $V$,
which satisfies the hereditary property: If $g \subseteq f$ and $f \in \sK$, then $g \in \sK$. A subhypergraph of a simplicial complex $\sK$ which is again a simplicial complex is called a \textit{subcomplex} of $\sK$. We call an element of $\sK$ a \textit{face} of $\sK$. In particular, we call a face of $\sK$ of size 1 or 2 by a \textit{vertex} or an \textit{edge}, respectively. We do not particularly distinguish an element $v\in V(\sK)$ from a singleton $\{v\}$ if there is no confusion.  Maximal faces of $\sK$ with respect to set inclusion are called \textit{facets}. The \textit{dimension of a face} $f$ is $ \vert f \vert -1$, and the \textit{dimension of a simplicial complex} $\sK$ is the maximum dimension of a face in $\sK$.

\medskip

For a simplicial complex $\sK$, we can correspond a topological space to $\sK$ which is unique up to homeomorphism. We call the space the \textit{realization} of $\sK$ and denote it by $ \vert  \vert \sK \vert  \vert $. For example, the power set of a set $V$ of size $d+1$, which we denote by $2^V$, is a simplicial complex whose realization is homeomorphic to the regular $d$-dimensional simplex. To emphasize that it is a simplicial complex, we use a different notation $\Delta(V)$ to denote $2^V$ and call it the \textit{abstract simplex} on $V$, or simply just an abstract simplex. Furthermore, when $ \vert V \vert =d+1$, we call it an abstract $d$-simplex. We call $\sK$ a \textit{simplicial $d$-sphere} or a \textit{simplicial $(d+1)$-ball} if $ \vert  \vert \sK \vert  \vert $ is homeomorphic to the $d$-dimensional unit sphere or the $(d+1)$-dimensional unit ball, respectively.

A simplicial complex $\sK$ is a \textit{PL $d$-sphere} if it is \textit{PL homeomorphic} to the boundary of an abstract $(d+1)$-simplex. Similarly, a simplicial complex $\sB$ is a \textit{PL $(d+1)$-ball} if $\sB$ is \textit{PL homeomorphic} to an abstract $(d+1)$-simplex. PL spheres and PL balls have nice properties so that they match better with our intuition than general simplicial spheres and balls. One of the properties is that the \textit{link} of a face $f$ of a PL $(d+1)$-ball $\sB$ defined as 
$$\lk_\sB(f)=\{g\in \sB: g \cup f \in \sB, g \cap f=\emptyset \} $$
is either a PL ball or a PL sphere. In the former case, we call $f$ a \textit{boundary face} of $\sB$, and in the latter case, $f$ is called an \textit{interior face} of $\sB$. The \textit{boundary complex} of a PL ball $\sB$, denoted by $\partial\sB$, is the subcomplex of $\sB$ which consists of all boundary faces of $\sB$. For details of PL topology, see \cite{piecewise_linear_hudson}.

\medskip 

Given a hypergraph $\HH$, we can correspond the simplicial complex $\sK(\HH)$ to $\HH$ by
$$\sK(\HH)=\bigcup_{f\in \HH}\Delta(f).$$
We can also correspond the \textit{facet hypergraph} $\FF(\sK)$, which consists of all facets of $\sK$, to a simplicial complex $\sK$.

A \textit{simplicial polytope} is a polytope whose faces are simplices. For a polytope $\mathbf{P}$, we can easily correspond the facet hypergraph $\FF(\mathbf{P})$ to $\mathbf{P}$ where each hyperedge is the set of all vertices of a facet $f$ of $\mathbf{P}$. Especially when $\mathbf{P}$ is simplicial, the simplicial complex $\sK(\mathbf{P})=\sK(\FF(\mathbf{P}))$, which we call the \textit{boundary complex} of $\mathbf{P}$, is a PL sphere, and this is our primary information we consider about a polytope.

For simple notations and conventions, when we consider the colorability and  transversals of $\FF(\sK)$ where $\sK$ is a simplicial complex, we directly refer to $\sK$ rather than $\FF(\sK)$ if there is no confusion. In the same manner, we directly refer to a polytope $\mathbf{P}$ for colorability and transversals, rather than referring to the corresponding simplicial complex $\sK(\mathbf{P})$ or the facet hypergraph $\FF(\mathbf{P})$.

\section{The transversal numbers and colorability for some special classes of simplicial spheres}\label{section_basic_examples}
In this section, we consider simplicial spheres of some special classes and investigate their transversal numbers and colorability. In particular, we show that every simplicial sphere considered here, except for odd cycles, is 2-colorable.

\subsection{2-dimensional simplicial spheres}\label{subsection_basic_3-polytope} We first consider the chromatic number of a simplicial 2-sphere.

\begin{prop}\label{prop_basic_3-polytope}
Simplicial 2-spheres are 2-colorable.
\end{prop}
\begin{proof}
Let $\sK$ be a simplicial 2-sphere. Since the graph $G=G(\sK)$ of $\sK$ is planar, by the four color theorem \cite{4-color-thm}, we have a proper coloring of $G$ which has color classes $V_1$, $V_2$, $V_3$ and $V_4$. For every face $f$ of $\sK$, the vertices of $f$ attain distinct colors from the coloring of $G$, since $f$ is a triangle. This implies that the vertex sets $W_1=V_1 \cup V_2$ and $W_2=V_3 \cup V_4$ are independent sets in the facet hypergraph $\FF(\sK)$, which form a 2-coloring of $\FF(\sK)$. 
\end{proof}
In the literature, Proposition \ref{prop_basic_3-polytope} is well-known and has many different proofs, see \cite[Section 2]{coloring_surface_jctb} for details. Note that Proposition \ref{prop_basic_3-polytope} does not lose generality when it is stated for simplicial 3-polytopes since every simplicial 2-sphere can be realized as the boundary of a 3-polytope, which is an easy consequence from Steinitz' theorem \cite{steinitz_thm} and Whitney's theorem \cite{whitney_thm}. 

\medskip 

Proposition \ref{prop_basic_3-polytope} implies that $\tau(\sK)\leq n/2$ where $\sK$ is a simplicial 2-sphere and $n$ is the number of vertices of $\sK$. We show that this bound is tight in Section \ref{section_plane_tight}.  

\subsection{Stacked balls and spheres}\label{subsection_basic_stacked}

A simplicial complex $\sB$ is called a \textit{stacked $(d+1)$-ball} if there are sequences of abstract $(d+1)$-simplices $\Delta_1, \dots, \Delta_m$ such that, with a  notation $\sB_k=\bigcup_{i=1}^k\Delta_i$, the intersection  $\sB_k\cap \Delta_{k+1}$ is an abstract
$d$-simplex whose maximal face is a $d$-dimensional face of both $\partial \sB_k$ and $\Delta_{k+1}$ for every $1\leq k \leq m-1$. A simplicial complex $\sK$ is called a \textit{stacked $d$-sphere} if $\sK=\partial \sB$ for some stacked $(d+1)$-ball $\sB$.

Stacked $(d+1)$-balls and stacked $d$-spheres form fundamental classes of simplicial complexes. They can be realized by polytopes, which are called \textit{stacked $(d+1)$-polytopes}. This in particular implies that they are PL $(d+1)$-balls and PL $d$-spheres, respectively. Stacked balls and spheres are closely related to neighborly spheres, see Section \ref{subsection_basic_single_element_ext} below.

\medskip 

The following is an easy observation.

\begin{prop}\label{prop_basic_stacked}
Let $d\geq 2$. For a given stacked $(d+1)$-ball $\sB$ on $n$ vertices and $\sK=\partial \sB$, we have $\tau(\sB)\leq \frac{1}{d+2}n$ and $\tau(\sK)\leq \frac{2}{d+2}n$. Furthermore, $\sB$ and $\sK$ are both 2-colorable.
\end{prop}
\begin{proof}
By the inductive construction of $\sB$, one can easily see that the graph of $\sB$, or equivalently the graph of $\sK$, is $(d+2)$-colorable. Since every facet of $\sB$ attains all $d+2$ colors from the coloring, choosing a color class of the minimum size gives a transversal of size at most $n/(d+2)$ for $\sB$. For a transversal of $\sK$, we choose the two smallest color classes. Then the union of the two color classes is a transversal of size at most $2n/(d+2)$ for $\sK$.
For both $\sB$ and $\sK$, the transversals we found and their complements form proper 2-colorings. 
\end{proof}

The following example gives the largest transversal ratio we know for general stacked balls and spheres.  

\begin{eg}\label{eg_stacked_lower}
Let $\sB^{d+1}_n$ be a stacked $(d+1)$-ball which has the facet hypergraph
$$\FF(\sB^{d+1}_n)=\{ \{1,2,\dots, d+2\},\{2,3,\dots, d+3\}, \dots, \{n-d-1, n-d,\dots  ,n\}\}.$$
When $n=(d+2)k$ for some $k \geq 1$, $\sB^{d+1}_n$ has the transversal number exactly $\frac{1}{d+2}n$. Since we already obtained the upper bound in Proposition \ref{prop_basic_stacked}, we only need to show that a transversal of $\sB^{d+1}_n$ has size at least $\frac{1}{d+2}n$. Note that there are $n-d-1=(d+2)(k-1)+1$ facets in $\sB^{d+1}_n$, and each vertex of $\sB^{d+1}_n$ can only traverse at most $(d+2)$ facets. Therefore, a transversal of $\sB^{d+1}_n$ requires at least 
$$\geq \left\lceil\frac{(d+2)(k-1)+1}{d+2}\right\rceil=k=\frac{n}{d+2}$$
vertices.

Now, let $\sK^d_n=\partial \sB^{d+1}_n$. We show that the stacked $d$-sphere $\sK^d_n$ has the transversal number exactly $\frac{2}{d+3}n$ when $n=(d+3)k$ for some $k\geq 1$. Note that $\sK^d_n$ has the facet hypergraph
$$\FF(\sK^d_n)=\{g\setminus \{w\} :g\in \FF(\sB^{d+1}_n), w\notin \{\max(g), \min(g)\} \}
\cup \{\{1,\dots, d+1\}, \{n-d,\dots, n\}\}.$$

We partition the vertex set $[n]$ into the subsets
$$\{1, 2,\dots, d+3\},\{d+4, d+5,\dots, 2(d+3)\},\dots, \{(k-1)(d+3)+1, \dots, k(d+3)\}.$$
Let us call the subsets $V_1, \dots, V_k$ in order. We show that the set $T$ which consists of all maximums and minimums of $V_1, \dots, V_k$ is a transversal of $\sK^d_n$. Suppose otherwise that there is a facet $f$ of $\sK^d_n$ which is not traversed by $T$. If $f$ is contained in $V_i$ for some $i\in [k]$, then $f$ should miss $\max(V_i)$ and $\min(V_i)$. Then, $f$ should consist of the middle elements of $V_i$ which is impossible. So $f$ needs to intersect two subsets $V_i$ and $V_{i+1}$ for some $i\in [k-1]$. Then, $f$ should miss two consecutive elements, which is also impossible. This shows that $\tau(\sK^d_n)\leq \frac{2}{d+3}n$.

We can also easily see that the induced subhypergraph $\FF(\sK^d_n)[V_i]$ requires at least two vertices to traverse all facets in the subhypergraph for each $i\in [k]$. Together with the upper bound, this shows that $\tau(\sK^d_n)= \frac{2}{d+3}n$.\qed 
\end{eg}
Example \ref{eg_stacked_lower} shows that the upper bound for stacked $(d+1)$-balls in Proposition \ref{prop_basic_stacked} is tight. However for stacked $d$-spheres, it only shows that the exact maximum transversal ratio is in somewhere between $\frac{2}{d+3}$ and $\frac{2}{d+2}$.

\begin{Q} \label{Question_stacked_sphere}
What is the following value
$$\sigma_d:=\limsup_{n \to \infty} \max\{\tau(\sK)/n: \textrm{$\sK$ is a stacked $(d-1)$-sphere on $n$ vertices}\}$$
for a fixed dimension $d$?
\end{Q}

\begin{remark}
Regarding Question \ref{Question_stacked_sphere}, the following two results were recently obtained by Minho Cho and Jinha Kim, as well as general results for higher dimensions \cite{cho2023transversal}. (i) They showed that $\sigma_3\geq 6/13$. (ii) Define 
\begin{align*}&\sigma_d^\textrm{path}
:=\\&\limsup_{n \to \infty} \max\{\tau(\partial\sB)/n: \textrm{$\sB$ is a stacked $d$-ball on $n$ vertices, $G^*(\sB)$ is a path}\}.
\end{align*}
They obtained $\sigma_3^\textrm{path}=3/7$. (In the definition of $\sigma_d^\textrm{path}$, $G^*(\sB)$ is the \textit{facet-ridge graph} of $\sB$, where vertices are facets of $\sB$ and two facets of $\sB$ are adjacent in $G^*(\sB)$ if and only if they share a codimension 1 face.)  The former result is based on computer experiments while the latter has a hand-written proof. 
\end{remark}

\subsection{Cyclic polytopes}\label{subsection_basic_cyclic}
The \textit{moment curve} $\gamma_d: \RR\rightarrow \RR^d$ in $\RR^d$ is defined as
$$ \gamma_d(t)=(t,t^2,\dots, t^d)\in \RR^d.$$
Given $n$ real numbers $t_1<t_2< \cdots<t_n$ where $n\geq d+1$, we define the \textit{cyclic polytope} $C_d(t_1, \dots, t_n)$ as $$C_d(t_1, \dots, t_n)=\conv\{\gamma_d(t_1),\dots, \gamma_d(t_n)\}.$$

In fact, the following criterion gives a full combinatorial description of $C_d(t_1, \dots, t_n)$. In particular, it shows that the combinatorial type of $C_d(t_1, \dots, t_n)$ is independent from the choice of real parameters $t_1, \dots, t_n$. By this, we use a simpler notation $C_d(n)$ and call it the \textit{cyclic $d$-polytope on $n$ vertices}. Furthermore, we label the vertices $\gamma_d(t_1),\dots, \gamma_d(t_n)$ of $C_d(n)$ by $1, \dots, n$ in order.

\begin{thm}[Gale's evenness criterion] \label{thm_gale}
$C_d(n)$ is a simplicial polytope. Furthermore, a $d$-subset $f$ of $[n]=\{1,\dots, n\}$ forms a facet of $C_d(n)$ if and only if  the set $\{k\in f:i<k<j \}$ has even size for every $i, j \notin f$.
\end{thm}

Using Gale's criterion, one can obtain the following fact.

\begin{prop}\label{prop_basic_cyclic_col}
$C_d(n)$ is 2-colorable for every $d\geq 3$. Furthermore, 
$$\tau(C_d(n))=\left\lceil \frac{n-d}{2} \right\rceil+1$$
when $d$ is even, and $\tau(C_d(n))=2$ when $d$ is odd.
\end{prop}

\begin{proof}
We can obtain a proper 2-coloring of $C_d(n)$ by alternating colors of vertices as we move from the vertex 1 to the vertex $n$. We can also easily see that when $d$ is odd, $\tau(C_d(n))=2$ and $\{1,n\}$ is a transversal of $C_d(n)$ of the minimum size.

\medskip 

We now compute the transversal number when $d$ is even. We first show that there are no transversals of size $\lceil\frac{n-d}{2} \rceil$. Suppose otherwise that there is such a transversal $T$. Then we have $ \vert T \vert \leq \frac{n-(d-1)}{2}$. We will lead a contradiction by finding a facet $f$ of $C_d(n)$ which is contained in the complement $R=[n]\setminus T$ and has size at least $\frac{n+(d-1)}{2}$.

Note that $T$ divides the vertices of $R$ into ``intervals" of consecutive vertices. Among those intervals, there are at most $\frac{n-(d-1)}{2}-1$ ``inner" intervals which do not have the ending vertices 1 and $n$. We can choose a subset $g\subseteq R$ of the maximum size such that for each inner interval $I$ the intersection $g \cap I$ consists of an even number of consecutive vertices. Note that $g$ only misses at most 1 vertex at each inner interval. Therefore, $g$ has the size at least
$$\frac{n+(d-1)}{2}-\left(\frac{n-(d-1)}{2}-1\right)=d.$$
We can choose a subset $f\subseteq g$ so that $ \vert f \vert =d$, and for each inner interval $I$, $f\cap I$ consists of an even number of consecutive vertices. By Gale's criterion, $f$ forms a facet which leads to a contradiction.

\medskip 

Finally, we find a transversal of size $\lceil \frac{n-d}{2}\rceil+1$. Let $I_1$ and $I_n$ be some intervals of consecutive vertices which contain 1 and $n$ resepectively. We further assume that $ \vert I_1 \vert + \vert I_n \vert =d-1$. We alternatingly choose vertices for a transversal, starting from the vertex right after $I_1$ before we arrive at $I_n$. More precisely, when $n-(d-1)$ is odd, we choose a transversal as
$$ I_1 \; *\; - \; * \; -\; \dots\;*\; -\; * \;  I_n,$$
where ``$*$" denotes a vertex chosen for a transversal, and ``$-$" denotes one of the other vertices. Using the same notation, when 
$n-(d-1)$ is even, we choose a transveral as
$$ I_1 \; *\; - \; * \; -\; \dots\;*\; -\; * \; *\; I_n.$$
This completes the proof.
\end{proof}

In particular, the fact $\tau(C_d(n)) > \lceil \frac{n-d}{2}\rceil$ when $d$ is even is already observed in \cite{holmsen_surrounding}. This fact is equivalent to saying that $C_d(n)$ satisfies the dual surrounding property $S^*(\lceil \frac{n-d}{2}\rceil)$. Using the Gale transform, this implies there exist arbitrarily large point sets in $\RR^{\overline{d}}$ in general position with respect to the origin with property $S(\lfloor \frac{\overline{d}}{2} \rfloor +1)$ in the ``primal" space, see \cite[Theorem 4]{holmsen_surrounding}.

\subsection{Neighborly spheres obtained from a cyclic polytope by one single element extension} \label{subsection_basic_single_element_ext}
A simplicial complex $\sK$ on the vertex set $V$ is called \textit{$k$-neighborly} if every $k$-subset of $V$ is a face of $\sK$. A simplicial $d$-sphere is called a \textit{neighborly $d$-sphere} if it is $\lfloor(d+1)/2\rfloor$-neighborly. Using Gale's evenness criterion \ref{thm_gale}, one can show that a cyclic polytope is a \textit{neighborly polytope}, that is, $\sK(C_d(n))$ is a neighborly sphere for all positive integers $d$ and $n$ with $n\geq d+1$.

Neighborly spheres form an interesting class of simplicial spheres because they satisfy extremal properties. For example, the upper bound theorem \cite{upper_bound_mcmullen,upper_bound_stanley} asserts that a neighborly sphere attains the maximum number of faces.

\medskip 

There is a well-known inductive method which given a neighborly sphere constructs a new neighborly sphere with an additional vertex.
To do that we need a more general concept than the stackedness we defined at Section \ref{subsection_basic_stacked}. We call a PL $d$-ball $\sB$ \textit{$k$-stacked} if all faces of $\sB$ of dimension at most $d-k-1$ are in $\partial \sB$. Note that $\sB$ is 1-stacked if and only if $\sB$ is stacked. 

\begin{prop}[Lemma 4.1 in \cite{novik_neighborly}] \label{prop_single_element_ext-gen}
Let $\sK$ be a neighborly PL $d$-sphere on $V$. Let $\sB$ be a $(\lfloor (d+1)/2\rfloor-1)$-neighborly (with respect to $V$) and $(\lfloor (d+1)/2\rfloor-1)$-stacked PL $d$-ball which is contained in $\sK$ as a subcomplex. Then for a new vertex $w\notin V$, the new simplicial complex $$(\sK\setminus \sB)\cup (\partial \sB* \{\emptyset, \{w\}\})$$ is a neighborly PL $d$-sphere.
\end{prop}
Also, check \cite[Theorem 1]{nine_vertices} for the converse for neighborly $4$-polytopes. We call the process described in Proposition \ref{prop_single_element_ext-gen} a \textit{single element extension}.

\medskip

In the literature, it seems that enumerating all neighborly $d$-spheres is a very challenging task even for small size already in dimension $d=3$, for example, see \cite{nine_vertices, nine_vertices_sphere}. Hence, it looks difficult to obtain full combinatorial characterizations of neighborly spheres and colorability results of neighborly spheres from them. However, it is still possible to obtain the results for spheres obtained by applying just one single element extension to a cyclic polytope. The resulting neighborly sphere shares many faces with the original cyclic polytope, and this is the key property to prove the following proposition.

\begin{prop}
Let $d\geq 2$. Let $\sK$ be a neighborly $d$-sphere which is obtained from a cyclic polytope $C_{d+1}(n)$ by applying one single element extension which adds a new vertex $w$. Then $\sK$ is 2-colorable.
\end{prop}

\begin{proof}
The case when $d=2$ is easily obtained from Proposition \ref{prop_basic_3-polytope}. So we assume $d\geq 3$. We identify the vertices of a cyclic polytope $C_{d+1}(n)$ with $1, \dots, n$ in order. We show how we can extend the alternating coloring $c: [n] \to \{R, B\}$ of $C_{d+1}(n)$ to a proper 2-coloring of $\sK$.

In fact, we only need to consider a facet which is newly added in the single element extension. Such a facet has the form $f\cup \{w\}$ where $ \vert f \vert =d\geq 3$ and $f$ is a face of some facet in $C_{d+1}(n)$. Among such faces $f$, again, it is enough to only consider monochromatic ones in the coloring $c$. By Gale's evenness criterion \ref{thm_gale} and the choice of $c$, we should have  $\vert f\setminus \{1,n\}\vert\leq 1$ otherwise
there is no way to extend $f$ to a facet of $C_{d+1}(n)$. This only leaves the case when $d=3$,  $\{1, n\}\subseteq f$ and $c(1)=c(n)$. Now, we color $w$ with the color different from $c(1)$. 
\end{proof}

In \cite{novik_neighborly}, Novik and Zheng constructed many neighborly spheres and counted the number of them. With their result (see \cite[Theorem 1.2]{novik_neighborly} and the proof), we can conclude that there are at least $2^{\Omega(n^{\lfloor d/2\rfloor})}$ neighborly $d$-spheres which are 2-colorable when $d$ is odd and $d\geq 5$.
\begin{remark}
In \cite{miyata}, Miyata and Padrol enumerated all neighborly oriented matroids of rank $r$ on $n$ elements for some small values of $r$ and $n$. Using the result, they tested many properties of oriented matroids, including the dual surrounding property $S^*(k)$. Their results particularly imply that every neighborly oriented matroid $\MM$ constructed in \cite{miyata} satisfies $\tau(\FF(\MM))\leq n/2$ where $\FF(\MM)$ denotes the facet hypergraph of $\MM$, which records incidences between the facets and elements of $\MM$, and $n$ is the number of elements of $\MM$. 
\end{remark}

\section{Constructions of \texorpdfstring{$(d+1)$-polytopes with transversal ratio $\frac{2}{d+2}$}{}}\label{section_plane_tight}
In this section we prove Theorem \ref{thm_plane_tight}. We first begin by introducing a notion which will be useful in the constructions of Theorem \ref{thm_plane_tight} and \ref{thm_trans}.

\begin{Def}
Let $\sK_1$ and $\sK_2$ be PL $d$-spheres, and let $f_i$ be a facet of $\sK_i$ for each $i \in [2]$. Given a bijection $\psi: f_1 \to f_2$, we define the \textit{connected sum} $\sK_1 \#_\psi \sK_2$ of $\sK_1$ and  $\sK_2$ with respect to $\psi$ as 
$$\sK_1 \#_\psi \sK_2=(\sK_1 \setminus \{f_1\})\sqcup  (\sK_2 \setminus \{f_2\})/\sim_\psi,$$
where $\sqcup$ denotes disjoint union and every proper subface $g_1$ of $f_1$ is identified with $\psi(g_1)$. 
\end{Def}

\begin{lemma}\label{lemma_connected_sum_sphere}
The connected sum $\sK_1\#_{\psi}\sK_2$ is also a PL $d$-sphere.
\end{lemma}

\begin{proof}
By Newman's theorem \cite[Theorem 4.7.21 (iii)]{om_book}, $\sK_i\setminus \{f_i\}$ is a PL $d$-ball for every $i\in [2]$. Since those two PL $d$-balls are attached along their entire boundaries, the connected sum is a PL $d$-sphere by \cite[Theorem 4.7.21 (ii)]{om_book}.
\end{proof}

\medskip 

Let $d \geq 2$ be a fixed integer. Let us denote the boundary complex of the abstract $(d+1)$-simplex on a vertex set $\{v_1, v_2, \dots, v_{d+2}\}$ by $\cS^d$, that is, let $$\cS^d=\partial \Delta(\{v_1, v_2, \dots, v_{d+2}\}).$$ Let $\CC^d$ be 
the boundary complex of the $(d+1)$-dimensional cross polytope, that is, $\CC^d$ is the $(d+1)$-fold join
$$\CC^d=\cS^0* \cdots * \cS^0.$$
Note that $\CC^d$ has a natural antipodal action.

Now, let us choose arbitrary facets $f$ and $g$ from $\CC^d$ and $\cS^d$, and let $\psi_{\cS}:f\to g$ be an arbitrary identification map. We denote $\CC^d \#_{\psi_{\cS}} \cS^d$ by $\CC^d_+$. Note that by symmetry, $\CC^d_+$ is independent from the choice of the facets $f$ and $g$, and the identification map $\psi_\cS$. The \textit{apex vertex} of $\CC^d_+$ is the vertex $w$ that was opposite to $g$ in $\cS^d$, and the \textit{base facet} of $\CC^d_+$ is the facet that was antipodal to $f$ in $\CC^d$. Since $\cS^d$ and $\CC^d$ have $d+2$ and $2(d+1)$ vertices respectively, and they share $d+1$ vertices in common at $\CC^d_+$, $\CC^d_+$ has
$$(d+2)+2(d+1)-(d+1)=2d+3$$
vertices.

\medskip 

Using the simplicial sphere $\CC^d_+$, we inductively construct simplicial $d$-spheres $\XX^d_k$ as follows. 
\begin{itemize}
\item Let $\XX^d_1=\cS^d$. We choose an arbitrary vertex of $\XX^d_1$ as the apex vertex of $\XX^d_1$. Obviously $\XX^d_1$ has $d+2$ vertices. 

\item When $k\geq 2$, let $\XX^d_k=\XX^d_{k-1}\#_{\psi^d_k} \CC^d_+$, where $\psi^d_k$ is an arbitrary bijection from a face $f_{k-1}$ which contains the apex vertex of $\XX^d_{k-1}$ to the base facet of $\CC^d_+$. We choose the apex vertex of $\CC^d_+$ as the apex vertex of $\XX^d_k$. This connected sum operation adds
$$(2d+3)-(d+1)=d+2$$
vertices to $\XX^d_{k-1}$ since $\CC^d_+$ has $2d+3$ vertices and shares $d+1$ vertices with $\XX^d_{k-1}$ in $\XX^d_k$. Using induction, we can conclude that $\XX^d_k$ has $k(d+2)$ vertices.
\end{itemize}
For our purpose, it does not matter which face $f_{k-1}$ and identification map $\psi^d_k$ we choose.

\begin{figure}[ht]
\centering
\includegraphics[width=10cm]{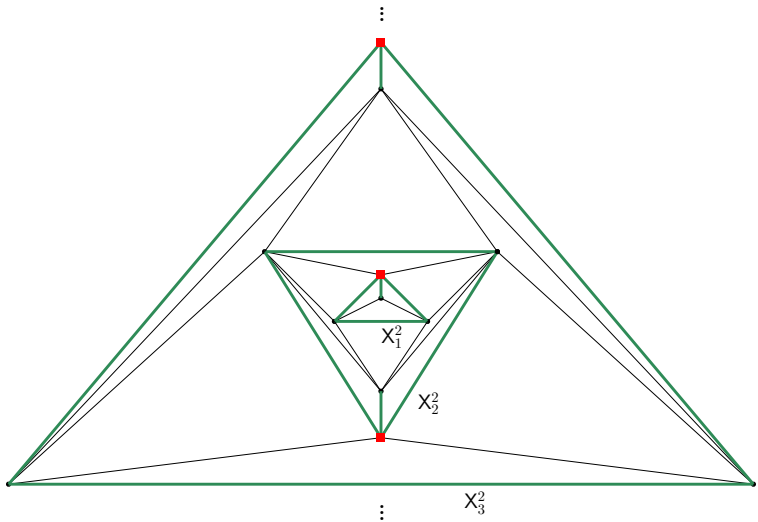}
\caption{A Schlegel diagram of $\XX^2_3$.
The three triangles with thick green edges surround Schlegel diagrams of the complexes $\XX^2_k$ for $k\in [3]$. Here, red squares describe the apex vertices of $\XX^2_k$.
}
\end{figure}	

\medskip 

Note that each of $\cS^d$ and $\CC^d_+$ can be realizable as the boundary of a $(d+1)$-dimensional polytope. 
Using these observations, the following lemma gives a stronger conclusion than Lemma \ref{lemma_connected_sum_sphere} that  $\XX^d_k$ can be realized as the boundary of a $(d+1)$-polytope.
\begin{lemma}\label{lemma_connected_sum_polytope}
Suppose that each of simplicial $d$-spheres $\sK_1$ and $\sK_2$ is realizable as the boundary of a $(d+1)$-polytope. Then the connected sum 
$\sK_1 \#_\psi \sK_2$ is realizable as the boundary of a $(d+1)$-polytope for any map $\psi$ from a facet of $\sK_1$ to a facet of $\sK_2$.
\end{lemma}

This is a special case of \cite[Lemma 3.2.4]{realization_space_polytope}. Since there are $k(d+2)$ vertices in $\XX_k^d$, it is enough to show the following claim to prove Theorem \ref{thm_plane_tight}.

\begin{claim}
We have $\tau(\XX^d_k)=2k$. Furthermore, there is no minimum transversal of $\XX^d_k$ containing a facet incident to the apex vertex of $\XX^d_k$.
\end{claim}

To prove the claim, we use induction on $k$. When $k=1$, by definition $\XX^d_1$ is the boundary of an abstract $(d+1)$-simplex. Hence, $\tau(\XX^d_1)=2$. Since every facet of $\XX^d_1$ has size $d+1 > 2$, the second claim also holds for this case.

\smallskip

Now let us assume that the claim holds for $\XX^d_{k-1}$ where $k\geq 2$. Let $n= \vert V(\XX^d_k) \vert $, that is, let $n=k(d+2)$. We denote by $\CC^d_k$ the complex $\CC^d_+$ which is glued at $\XX^d_{k-1}$ to form $\XX^d_k$. Note that there is a natural antipodal action on non-apex vertices of $\CC^d_k$.

We first show that $\tau(\XX^d_k)\leq \frac{2}{d+2}n =2k$. In fact, it is easy to show that the graph of $\XX^d_k$ is $(d+2)$-colorable by induction: Suppose a proper $(d+2)$-coloring was already given to the graph of $\XX^d_{k-1}$. When we glue $\CC^d_k$ at $\XX^d_{k-1}$, the colors of $d+1$ vertices of $\CC^d_k$ in the base facet, which are necessarily distinct, are already determined by the precoloring of $\XX^d_{k-1}$ via identification. We color each of the remaining non-apex vertices of $\CC^d_k$ with the same color of the antipodal vertex of it in the base facet. We complete the extension by assigning to the apex vertex of $\CC^d_k$ the color which was not used on the base facet of $\CC^d_k$. The union of any two among $d+2$ color classes of the coloring is a transversal of $\XX^d_k$, and choosing a minimum-sized union among them gives the bound.

\smallskip 

Next, we show that a transversal of $\XX^d_k$ requires at least $\frac{2}{d+2}n=2k$ vertices. Suppose that $T$ is a minimum transversal of $\XX^d_k$. Let 
$f_{k-1}^-$ be the set of non-apex vertices of $\CC^d_k$ which are not in the base facet (which is identified with $f_{k-1}$) of $\CC^d_k$. Let $\FF$ be the set of facets of $\XX^d_k$ which are contained in $f_{k-1}\cup f_{k-1}^-$.  Note that $\sK(\FF)$ is isomorphic to the boundary complex of a cross polytope with a pair of antipodal facets removed.  There are 3 cases to consider.

\begin{enumerate}
\item When $T$ does not contain any vertex of $f_{k-1}$: In this case, we need at least $2(k-1)-1=2k-3$ vertices to traverse all facets of $\XX^d_{k-1}\setminus \{f_{k-1}\}$ by induction. To traverse all facets in $\FF$, we need all vertices of $f_{k-1}^-$, which has size $d+1$. Hence,

$$\vert T \vert \geq (2k-3)+(d+1) \geq (2k-3)+3=2k.$$
Note that this case never happens when $d\geq 3$, because we already have $\tau(\XX^d_k)\leq 2k$, but for $d\geq 3$, the above bound becomes $ \vert T \vert \geq (2k-3)+(d+1) \geq 2k+1$, contradicting $T$ is a minimum transversal.

\item When $T$ contains a vertex of $f_{k-1}$ but not all vertices: Note that in this case, $T\cap V(\XX^d_{k-1})$ is also a transversal of $\XX^d_{k-1}$. This requires at least $2(k-1)=2k-2$ vertices for $T\cap V(\XX^d_{k-1})$. Since $T$ does not contain $f_{k-1}$, there is at least one facet in $\FF$ not covered by $T \cap V(\XX_{k-1}^d)$, and $T$ must contain some vertex $w \in f_{k-1}^-$ in order to cover that facet. Additionally, $T$ must have at least one further vertex to cover the facet $\{v\} \cup f_{k-1}^- \setminus \{w\}$, where $v$ is the apex vertex of $\XX_k^d$. Hence,
$$ \vert T \vert \geq (2k-2)+2=2k.$$

\item When $T$ uses all vertices of $f_{k-1}$:  Again, $T\cap V(\XX^d_{k-1})$ is a transversal of $\XX^d_{k-1}$. By the second claim for $\XX^d_{k-1}$, $T\cap V(\XX^d_{k-1})$ is not a minimum transversal of $\XX^d_{k-1}$. This gives $\vert T\cap V(\XX^d_{k-1})\vert \geq 2(k-1)+1=2k-1$. Note that $T$ should use the apex vertex of $\XX^d_k$ to satisfy the minimality. Therefore, we have
$$\vert T\vert \geq (2k-1)+1=2k.$$
\end{enumerate}

The three scenarios described above show that if $T$ is a transversal of $\XX_k^d$ of size $2k$ then $T$ cannot contain any facet which includes the apex vertex $v$, since in the first case $T$ does not contain $v$ and in the others $T$ contains at most two vertices from $f_{k-1}^- \cup \{v\}$. This completes the proof of the claim. \qed 

\medskip

Particularly for the 2-dimensional case, 
the construction of Theorem \ref{thm_plane_tight} is not the only one  with transversal ratio exactly $1/2$. 
The boundary complex of the regular icosahedron is another example. The reader may verify this.  

Note that a simplicial 2-sphere $\sK$ with transversal ratio exactly $1/2$ has a strong symmetry, for example, every proper 4-coloring of the graph of $\sK$ has color classes of equal size. It might be interesting to investigate such simplicial 2-spheres further.

\begin{Q}\label{question_flip}
Can we characterize all simplicial 2-spheres with transversal ratio exactly $\frac{1}{2}$? What properties do they have?
\end{Q}

\section{Constructions of 3-spheres with transversal ratios larger than \texorpdfstring{$\frac{1}{2}$}{}}\label{section_non-2-colorable}
In this section, we prove Theorem \ref{thm_trans}. The proof relies on the following computational result.

\begin{thm}\label{thm_construction}
There is a 4-uniform hypergraph $\NN_{21}$ on 21 vertices such that $\sK(\NN_{21})$ is obtained by removing a single facet $\hat g$ from a neighborly PL 3-sphere $\sK_{21}$ and $\tau(\NN_{21})=11$.
\end{thm}

\subsection{An infinite construction based on Theorem \ref{thm_construction}}

Assuming Theorem \ref{thm_construction}, we first show how we can obtain an infinite construction which satisfies the conclusion of Theorem \ref{thm_trans}. Similarly with Section \ref{section_plane_tight}, we inductively construct simplicial 3-spheres using connected sums. Let $f$ be an arbitrary facet of $\CC^3$, the boundary complex of the 4-dimensional cross polytope which has 8 vertices, and $\psi_\sK :f\to \hat g$ be an arbitrary identification map for the facet $\hat g$ from Theorem \ref{thm_construction}.

We denote $\CC^3 \#_{\psi_\sK} \sK_{21}$ by $\CC_{\sK}$. Note that $\CC_{\sK}$ has $8+21-4=25$ vertices. Also note that by symmetry the isomorphism type of $\CC_\sK$ is independent from the choice of $f$ and $\psi_\sK$. We call the facet of $\CC_\sK$ which was the antipodal facet to $f$ in $\CC^3$ the \textit{base facet} of $\CC_\sK$.

\medskip 

We inductively construct 3-dimensional simplicial complexes $\YY_k$ as follows. 
\begin{itemize}
\item Let $\YY_1=\sK_{21}$. So $\YY_1$ has 21 vertices. We choose $\hat g$ as the base facet of $\YY_1$.

\item When $k\geq 2$, let $\YY_k=\YY_{k-1}\#_{\varphi_k} \CC_\sK$, where $\varphi_k$ is an arbitrary bijection from the base facet of $\YY_{k-1}$ to the base facet of $\CC_\sK$. We choose a facet from $\CC_\sK$ which was neither the base facet nor a facet in the copy of $\sK_{21} \setminus \{\hat{g}\}$ to be the base facet of $\YY_k$.
\end{itemize}
\begin{figure}[ht]
\centering
\includegraphics[width=9cm]{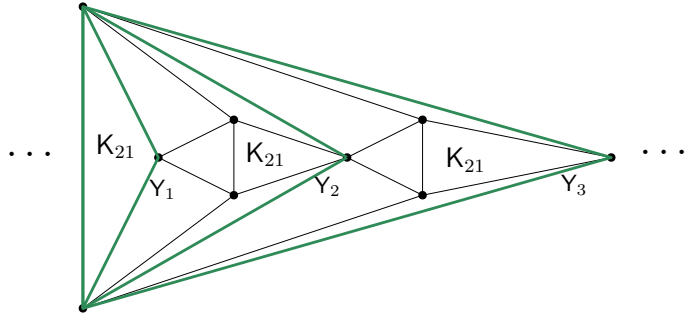}
\caption{A 2-dimensional illustration of the complexes $\YY_k$. 
The triangles with thick green edges surround Schlegel diagrams of the complexes $\YY_k$ for $k\in [3]$.} 
\end{figure}

Every time we conduct the connected sum operation with $\CC_\sK$ to $\YY_k$, we introduce $25-4=21$ additional vertices to $\YY_k$. So, $\YY_k$ has $21k$ vertices. This naturally partitions the vertex set $V(\YY_k)$ into $V_1, \dots, V_k$, where $V_i$ is the set of vertices introduced in the $i$th step. When $k\geq 2$, the induced subcomplex $\YY_k[V_i]$ has a facet hypergraph isomorphic to $\NN_{21}$. Hence, by Theorem \ref{thm_construction}, we need at least 11 vertices for each $V_i$ to traverse all facets of $\YY_k[V_i]$, which implies that the transversal ratio of $\YY_k$ is at least $11/21$. 

\smallskip

Since $\sK_{21}$ and $\CC^3$ are all PL spheres, each $\YY_k$ is a PL 3-sphere by Lemma \ref{lemma_connected_sum_sphere}.
This completes the proof of Theorem \ref{thm_trans}. \qed

\subsection{The experimental construction of \texorpdfstring{$\sK_{21}$}{}}

We now explain how we constructed the neighborly PL 3-sphere $\sK_{21}$ in Theorem \ref{thm_construction}.  The exact computational results are stated in Appendix \ref{section_appendix}.

Neighborly spheres have a large number of facets relative to the number of vertices, which might give a better chance to find a simplicial sphere of a larger transversal ratio. We simplify and restate Proposition \ref{prop_single_element_ext-gen} for our case, which describes an inductive method to construct neighborly spheres. Note that we do not particularly state the PL condition since every simplicial $d$-sphere is PL for $d\leq 3$ (see the paragraph after \cite[Definition 4.7.20]{om_book}).

\begin{prop}[Restatement of Proposition \ref{prop_single_element_ext-gen} when $d=3$] \label{prop_single_element_ext}
Let $\sK$ be a neighborly $3$-sphere on $V$. Let $\sB$ be a subcomplex of $\sK$ on $V$ which is a stacked $3$-ball. Then for a new vertex $w\notin V$, the new simplicial complex 
\[(\sK\setminus \sB)\cup (\partial \sB* \{\emptyset, \{w\}\})\] 
is a neighborly $3$-sphere.
\end{prop}
Recall that we call the process of Proposition \ref{prop_single_element_ext} a \textit{single element extension}.

\medskip 

Here is how we applied Proposition \ref{prop_single_element_ext}: Starting from the boundary complex of a cyclic polytope $\sK(C_4(7))$, we applied many single element extensions and obtained a neighborly sphere of size one larger than the previous one at a time. This resulting
neighborly sphere at each step was chosen to be optimal in a parameter $\varepsilon(\cdot)$ among all spheres of the same size randomly sampled in that step. We iterated this process until we got the complex $\sK_{21}$.

\medskip 

The purpose of the $\varepsilon$ parameter is to choose a neighborly sphere which is more likely to have a large transversal number. It is defined as follows.

\begin{Def}[The degree of an edge (of size 2) and $\varepsilon(\sK)$] \label{def_edge_degree}
For a simplicial sphere $\sK$ and an edge $e$ of $\sK$, $\deg_\sK(e)$ counts the number of facets of $\sK$ which contains $e$ as a subset. We define a parameter $\varepsilon(\sK)$ of a simplicial sphere $\sK$ on $n$ vertices as the sum of the $n$ largest values of the degrees of edges of $\sK$.
\end{Def}

\begin{eg}[The $\varepsilon$ value for cyclic 4-polytopes] \label{eg_cyclic_edge_degree} 
By Gale's evenness criterion (Theorem \ref{thm_gale}), one can see that a subset of the form $\{i, i+1\}$ modulo $n$ has degree $n-2$ in $C_4(n)$, which is the maximum possible value for the degree of an edge in a simplicial 3-sphere. Since there are $n$ such 2-subsets, we have $\varepsilon(\sK(C_4(n)))=n(n-2)$. \qed
\end{eg}

In the construction, we kept the $\varepsilon$ value \textit{as low as possible}. We briefly give a justification for this heuristic. The face numbers of a neighborly $d$-sphere are determined by its initial condition being $\lfloor (d+1)/2 \rfloor$-neighborly and the Dehn–Sommerville equations (\cite[Theorem 8.21]{lectures_on_polytopes_book}, check also the remark before Corollary 8.31 in \cite{lectures_on_polytopes_book}). So it is impossible to distinguish one neighborly 3-sphere from another by just comparing vertex degrees. Rather, the degree of an edge is more suitable to distinguish neighborly 3-spheres. 

Recall that our purpose is to obtain a simplicial sphere on $n$ vertices with the transversal number larger than $n/2$. A necessary condition of this is that the simplicial sphere is not 2-colorable. Here is a simple observation on 2-colorability of a simplicial sphere.
\begin{prop} \label{prop_edge_cover}
For a simplicial sphere $\sK$, $\FF(\sK)$ is 2-colorable if and only if there is a bipartite subgraph $G$ of the graph (or the 1-skeleton) of $\mathsf{K}$ such that for every facet $f$ of $\mathsf{K}$ there is an edge $e\in E(G)$ such that $e\subseteq f$.
\end{prop}
In other words, in order to have 2-colorability of $\sK$, $\sK$ needs to have a bipartite graph which covers all facets of $\sK$. The chance that there is such a covering bipartite graph might become higher when we can choose edges with large degrees.
One example is a cyclic polytope as described in Example \ref{eg_cyclic_edge_degree}. In fact, those edges in Example \ref{eg_cyclic_edge_degree}
cover all facets of $C_4(n)$ even without using the edge $\{1, n\}$, and in this sense 
cyclic polytopes are extremal.
As opposed to cyclic polytopes, we keep the $\varepsilon$ value low as possible, in order to find a neighborly sphere with different characteristics from cyclic polytopes, which is not 2-colorable, or even has the transversal number larger than $n/2$. 

\begin{remark}\label{remark_construction}
We also constructed a non-2-colorable neighborly 3-sphere on 13 vertices with a similar method. It should also be  noted that optimizing the $\varepsilon$ parameter does not look essential in finding a desirable construction, but it seems to make the process more efficient. 
\end{remark}

\subsection{Non-realizability of \texorpdfstring{$\sK_{21}$}{}} If $\sK_{21}$ can be realized as the boundary of a polytope, then, by using Lemma \ref{lemma_connected_sum_polytope} rather than Lemma \ref{lemma_connected_sum_sphere}, 
we can show that the constructions we obtained in this section can be all realizable by polytopes. Unfortunately, $\sK_{21}$ is not realizable, so it does not have direct implications to the $k$-surrounding property $S(k)$. This observation is based on computational constructions by Miyata and Padrol \cite{miyata} and the following proposition. The following proposition and its proof are essentially from \cite{nine_vertices}.

\begin{prop} \label{prop_realizable_neighborly}
Let $\sK_1$ and $\sK_2$ be two neighborly 3-spheres where $\sK_2$ is obtained from $\sK_1$ by a single element extension which adds a new vertex $w$. Suppose that $\sK_2$ is realizable as the boundary of a polytope. Then $\sK_1$ is also realizable by a polytope. 
\end{prop}
As we mentioned before Proposition \ref{prop_single_element_ext}, every simplicial $d$-sphere is PL for $d\leq 3$. Hence, we do not particularly state the PL condition in Proposition \ref{prop_realizable_neighborly}.

We use the following lemma to prove Proposition \ref{prop_realizable_neighborly}.
\begin{lemma} \label{lemma_inv_single_element_extension} 
Let $\sL_2$ be a neighborly 3-sphere on a vertex set $V\cup \{w\}$. If there is another simplicial 3-sphere $\sL_1$ on $V$ which has all faces of $\sL_2[V]$, then the facets in $\sL_1\setminus \sL_2[V]$ form a stacked 3-ball.
\end{lemma}
For a set of simplices $\sK$, denote the number of $i$-dimensional faces of $\sK$ by $f_i(\sK)$. For a neighborly sphere $\sK$, recall that $f_i(\sK)$ is completely determined by $f_0(\sK)$ using the Dehn-Sommerville equations \cite[Theorem 8.21]{lectures_on_polytopes_book}. Especially when $\sK$ is a neighborly 3-sphere,
\begin{align}
f_2(\sK)=2\binom{f_0(\sK)-1}{2}-2, \textrm{ and } f_3(\sK)=\binom{f_0(\sK)-1}{2}-1.   \label{eqn_face_number_neighborly}
\end{align}

\begin{proof}[Proof of Lemma \ref{lemma_inv_single_element_extension}]

Let $n=\vert V \vert$. We first count how many new faces should be added to $\sL_2[V]$ in order to form $\sL_1$. Let $\mathcal{N}$ be the set of those new faces, that is, let $\mathcal{N}=\sL_1 \setminus \sL_2[V]$.
Note that $\sL_1$ is also neighborly since the 1-skeleton of $\sL_2[V]$ is complete and it is contained in $\sL_1$. In particular, $f_0(\mathcal{N})=f_1(\mathcal{N})=0$. So we only consider 2- and 3-faces.

Let $\mathcal{S}$ be the set of faces in $\sL_2$ which contains $w$. Note that each face of $\mathcal{S}$ corresponds to a face in the link
$$\lk_{\sL_2}(w):=\{g \in \sL_2: g\cup \{w\} \in \sL_2,\, w\notin g \}$$ which is 1 dimension lower. Since $\sL_2$ is a PL sphere, every link in $\sL_2$ is a simplicial sphere (see \cite[Theorem 4.7.21 (iv)]{om_book}). In particular,
$\lk_{\sL_2}(w)$ is a simplicial 2-sphere which contains all the other vertices of $\sL_2$ other than $w$, so 
\begin{align}
f_2(\mathcal{S})=3n-6, \textrm{ and } f_3(\mathcal{S})=2n-4. \label{eqn_link_face_number}
\end{align}
Hence, we have
\begin{align*}
& f_2(\mathcal{N})=f_2(\sL_1)-f_2(\sL_2[V])=f_2(\sL_1)-(f_2(\sL_2)-f_2(\mathcal{S}))\\
\stackrel{(\ref{eqn_face_number_neighborly}), (\ref{eqn_link_face_number})}{=}& \left(2\binom{n-1}{2}-2\right)-\left(2\binom{n}{2}-2\right)+(3n-6)=3n-6-2(n-1)=n-4,
\end{align*}
and similarly
\begin{align*}
& f_3(\mathcal{N})=f_3(\sL_1)-(f_3(\sL_2)-f_3(\mathcal{S}))\\
\stackrel{(\ref{eqn_face_number_neighborly}), (\ref{eqn_link_face_number})}{=}& \left(\binom{n-1}{2}-1\right)-\left(\binom{n}{2}-1\right)+(2n-4)=2n-4-(n-1)=n-3.
\end{align*}

Let $G$ be the graph which has the 3-faces of $\mathcal{N}=\sL_1 \setminus \sL_2[V]$ as vertices, and where a pair of 3-faces are adjacent when they share a common 2-face. Note that the conclusion is equivalent to $G$ being a tree.
Let $N$ (resp. $S$) be the complement of the geometric realization $\vert\vert\sL_2[V]\vert\vert$ of $\sL_2[V]$ as a subspace of a geometric realization of $\sL_1$ (resp. $\sL_2$). 
Let $\widetilde H_i(X)$ denote the $i$-th reduced homology group of a space $X$ and let $\widetilde H^i(X)$ denote the $i$-th reduced cohomology group of $X$.
Since $\sL_2[V]$ is a finite simplicial complex, $\vert\vert\sL_2[V]\vert\vert$ is compact and locally contractible \cite[Proposition A.4]{hatcher2002algebraic}, so by Alexander Duality \cite[Theorem 3.44]{hatcher2002algebraic}, $\widetilde H_0(N)$ is isomorphic to $\widetilde H^2(\vert\vert\sL_2[V]\vert\vert)$, which is isomorphic to 
$\widetilde H_0(S)$, which is the trivial group, since $S$ is the interior of a cone over a 2-sphere.  Hence, $N$ is path connected, and therefore we can always choose a path between two interior points from different facets of $N$. This corresponds to a path in $G$ since $\mathcal{N}$ only has 2- and 3-faces, and 2-faces of $\mathcal{N}$ correspond to edges in $G$. Thus, $G$ is a connected graph with $n-3$ vertices and $n-4$ edges, so $G$ is a tree, which means that the facets of $\mathcal{N}$ forms a stacked 3-ball.  

\end{proof}

\begin{proof}[Proof of Proposition \ref{prop_realizable_neighborly}]
Let $\mathbf{P}_2$ be a realization of $\sK_2$ as a polytope. For simplicity, we identify the vertex sets of $\sK_2$ and $\mathbf{P}_2$, and denote $V(\sK_2)\setminus \{w\}$ by $V$, that is, $V=V(\sK_1)$. Let $\mathbf{P}_1=\conv (V)$. We claim that $\sK(\mathbf{P}_1)=\sK_1$.

Note that $\sK_1 \cap \sK_2 = \sK_2[V]$: The inclusion $\subseteq$ is clear because every face of $\sK_1$ is contained in $V$. The other inclusion $\supseteq$ also holds because $\sK_2$ is obtained from $\sK_1$ by a single element extension. Also note that $\sK(\mathbf{P}_1) \cap \sK(\mathbf{P}_2) = \sK(\mathbf{P}_2)[V]$: $\subseteq$ holds by the same reason. $\supseteq$ also holds because every face of $\mathbf{P}_2$ whose vertices are in $V$ is also extreme in $\mathbf{P}_1$. As a conclusion, we get
\begin{align}
\sK_1 \cap \sK_2 = \sK_2[V]=\sK(\mathbf{P}_2)[V]= \sK(\mathbf{P}_1) \cap \sK(\mathbf{P}_2). \label{eqn_polytope_sphere_same}
\end{align}
It is enough to show that this identity extends to the rest of $\sK_1$ and $\sK(\mathbf{P}_1)$.

By Lemma \ref{lemma_inv_single_element_extension}, the facets of both $\sK_1 \setminus \sK_2$ and $\sK(\mathbf{P}_1) \setminus \sK(\mathbf{P}_2)$ are stacked 3-balls.
Every stacked 3-ball has a vertex whose link is a triangle, namely the last vertex added in the stacking, so by induction, a stacked 3-ball is completely determined by its boundary. 
The stacked 3-balls $\sK_1 \setminus \sK_2$ and $\sK(\mathbf{P}_1) \setminus \sK(\mathbf{P}_2)$ have the same boundary,
so $\sK_1 \setminus \sK_2 = \sK(\mathbf{P}_1) \setminus \sK(\mathbf{P}_2)$, 
and therefore $\sK_1 = \sK(\mathbf{P}_1)$.
\end{proof}

By Proposition \ref{prop_realizable_neighborly}, if $\sK_{21}$ were realizable, then $\sK_{11}$, the intermediate neighborly sphere with 11 vertices, should be also realizable by a neighborly 4-polytope. However, while we have $\varepsilon(\sK_{11})=74$, the minimum among $\varepsilon$ values of all neighborly oriented matroids of rank 5 with 11 elements, which in particular include all neighborly 4-polytopes, is $81$. This computation was done by a computer. We went through the supplementary data file ``\texttt{neighborly11\_5.zip}"\footnote{This file is available at \url{https://sites.google.com/site/hmiyata1984/neighborly_polytopes}.} of \cite{miyata} and calculated $\varepsilon$ values for each facet list. This shows that $\sK_{21}$ is not realizable. \qed

\begin{Q}\label{question_realizable}
For every integer $d\geq 3$, is there a $(d+1)$-polytope $\mathbf{P}$ such that $\tau(\mathbf{P})> \vert V(\mathbf{P}) \vert /2$?
\end{Q}

\section{An upper bound on the chromatic number of simplicial spheres} \label{section_upper_bound}
In this section, we prove Theorem \ref{thm_main_upper_bound}, that is, we prove that if $\mathbf{P}$ is a $(d+1)$-dimensional simplicial polytope on $n$ vertices, then $\chi(\mathbf{P}) \in O(n^{\frac{\lceil d/2\rceil-1}{d}})$. Later we generalize it for simplicial spheres (see Theorem \ref{thm_main_upper_bound_simplicial_sphere}) in Subsection \ref{subsection_upper_bound_embeddable}.

\medskip 

For Theorem \ref{thm_main_upper_bound}, we prove the following more general statement. A hypergraph $\HH$ is called \textit{linearly $d$-embeddable} if there is a linear embedding $\psi: \vert  \vert \sK(\HH) \vert  \vert \to \RR^d$.

\begin{thm}\label{thm_linearly_embed_upper_bound}
Let $\HH$ be a $(d+1)$-uniform hypergraph on $n$ vertices which is linearly $d$-embeddable. Then, we have
\[\chi(\HH) \in O(n^{\frac{\lceil d/2\rceil-1}{d}}).\]
\end{thm}
We first show that Theorem \ref{thm_linearly_embed_upper_bound} implies Theorem \ref{thm_main_upper_bound}. 

\begin{proof}[Proof of Theorem \ref{thm_main_upper_bound}]
Let $\mathbf{P}$ be a polytope, and $\HH$ be a $(d+1)$-uniform hypergraph obtained by removing one hyperedge of $\FF(\mathbf{P})$. By using an appropriate Schlegel diagram of $\mathbf{P}$, we can easily see that $\HH$ is linearly $d$-embeddable. By Theorem \ref{thm_linearly_embed_upper_bound}, $\chi(\HH)\in O(n^{\frac{\lceil d/2\rceil-1}{d}})$, which implies $\chi(\mathbf{P})\in O(n^{\frac{\lceil d/2\rceil-1}{d}})$.
\end{proof}

Now we begin to prove Theorem \ref{thm_linearly_embed_upper_bound}. The only implication from linear embeddability we use is the following upper bound theorem by Dey and Pach \cite[Theorem 2.1]{upper_bound_pach}.

\begin{thm} \label{thm_upper_bound_pach}
Let $\HH$ be a $(d+1)$-uniform hypergraph on $n$ vertices which is linearly embeddable into $\RR^d$. Then we have $ \vert \HH \vert <(d+1)n^{\lceil d/2 \rceil}$.
\end{thm}

We also need the following lemmas. 
A graph version of the first lemma is well-known, 
see for example \cite[Theorem 3.2.1]{probabilistic_method_book}. Our lemma can be proved in a similar way using the alteration method.

\begin{lemma}\label{lemma_alteration}
Let $\HH$ be a $(d+1)$-uniform hypergraph on $n$ vertices such that $ \vert \HH \vert \leq c n^k$. Then $\alpha(\HH)\geq \beta_{d,c}  n^{\frac{(d+1)-k}{d}}$, where $\beta_{d,c}$ depends only on $d$ and $c$.
\end{lemma}
For completeness, we include a modified proof here. 
\begin{proof}
Let $S \subseteq V(\HH)$ be a random subset obtained by including each vertex in $S$ independently with a fixed probability $p$ determined later. Let $X= \vert S \vert $ and $Y$ be the number of hyperedges of $\HH$ contained in $S$.
Then removing one vertex from each hyperedge $f \subseteq S$ of $\HH$ leaves an independent set of size at least $X-Y$. 
For each $f\in \HH$, let $Y_f$ be the indicator random variable for the event $f \subseteq S$. Then  $E[Y_f]=p^{d+1}$. Using 
linearity of expectation, we get
$$E[Y]=\sum_{f \in \HH} E[Y_f]=
\vert \HH \vert  p^{d+1} \leq cn^kp^{d+1},$$
which implies
$$E[X-Y]=E[X]-E[Y]\geq np-cn^kp^{d+1}=np(1-cn^{k-1}p^d)=:f(p).$$

In the domain $0\leq p \leq 1$, we get the maximum of $f(p)$ at $p_0=1/\sqrt[d]{c(d+1)n^{k-1}}$, and 
$$f(p_0)=\frac{d}{(1+d)^\frac{d+1}{d}c^{\frac{1}{d}}} n^{\frac{(d+1)-k}{d}}$$ is a lower bound on $\alpha(\HH)$. \end{proof}

\medskip 

For a function $f:\RR \to \RR$ and a real number $n$, let $s_f(n)$ be the smallest non-negative integer such that $f^{s_f(n)}(n)< 1$, that is, iterated applications of $f$ on $n$ for $s_f(n)$ times makes the value less than 1. If there are no such numbers, we set $s_f(n)=\infty$.

\begin{lemma}\label{lemma_iteration}
Let $t<1$ and $\beta$ be positive reals. 
Define a function $h:\RR_{\geq 0} \to \RR_{\geq 0}$ as $h(x)= \max(x- \beta x^t,0) $. 
Suppose $(n=n_0, n_1, \dots, n_M)$ is a sequence of non-negative real numbers with $M\geq 1$ such that 
\begin{itemize}
\item $n_i \geq 1$ for every $0 \leq i \leq M-1$, \item $n_M< 1$, and 
\item $n_{i+1} \leq h(n_i)$ for every $i\in [M-1]$. 
\end{itemize}
Then, we have
$$
M \leq \frac{1}{\beta (1-t)}n^{1-t}+1.
$$
\end{lemma}

\begin{proof}
Let $g(x)=\frac{x^{1-t}}{\beta (1-t)}+1$. We first show that $s_h(x) \leq g(x)$ for every $x \geq 0$. 
For a contradiction, suppose it fails for some non-negative values of $x$, that is, we suppose
the set $F=\{x\in \RR_{\geq 0} :s_h(x)>g(x)\}$ is non-empty. Let $y_0=\inf F$.

We claim that $y_0> \beta^{\frac{1}{1-t}}$. Observe the followings:
\begin{enumerate}[(i)]
\item $h(x)=0$ for $x \in [0, \beta^{\frac{1}{1-t}}]$. 
\item $g(x)$ is strictly increasing and $g(x)\geq 1$ for $x \geq 0$.
\item 	$h(x)= x- \beta x^t$ and $h(x)$ is strictly increasing for $x\geq \beta^{\frac{1}{1-t}}$.

\end{enumerate}
By (i) and continuity of $h$, we know that $s_h(x)\leq 1$ for $x\in [0, \beta^{\frac{1}{1-t}}+\delta]$ where $\delta>0$ is sufficiently small. By (ii), we know that $g(x) \geq 1$ on the same interval $[0, \beta^{\frac{1}{1-t}}+\delta]$. Combining these two, we have that $g(x) \geq s_h(x)$ for $x\in [0, \beta^{\frac{1}{1-t}}+\delta]$, which implies $$y_0 \geq \beta^{\frac{1}{1-t}}+\delta >  \beta^{\frac{1}{1-t}}>0.$$

By this and (iii), we have $h(x)=x- \beta x^t$, $x>0$ and $h(x)>0$ for every $x \in F \cup \{y_0\}$. In particular, these imply that $y_0>h(y_0)$. By using continuity of $h$, we can choose $x_0 \in F$ near $y_0$ such that $h(x_0)<y_0$, which implies $h(x_0)\notin F$, that is, 
$$s_h(h(x_0))\leq g(h(x_0)).$$

Since $h(x_0)>0$, the function $g(x)$ is well-defined and continuous on $[h(x_0), x_0]$ and differentiable on $(h(x_0), x_0)$. Hence, by the mean value theorem, there is $\theta \in (h(x_0),x_0)$ such that
\begin{align*}
g(x_0)-g(h(x_0))&=g(x_0)-g(x_0-\beta x_0^t)\\
&=\beta x_0^t g'(\theta)=\frac{\beta x_0^t}{\beta \theta^t}=\left( \frac{x_0}{\theta} \right)^t \geq 1.
\end{align*}
Putting these together, we obtain
\[
s_h(x_0) \leq 1+s_h(h(x_0))
\leq 1+g(h(x_0))
\leq g(x_0),
\]
contradicting $x_0 \in F$.

\smallskip

Next, we show that $M\leq s_h(n)$. With  the inequality $s_h(n) \leq g(n)$, this gives the desired result. It is enough to show this by $h^i(n)\geq n_i$ for every $0\leq i \leq M$ using induction on $i$. The claim obviously holds for $i=0$. Suppose $h^i(n) \geq n_i$ holds for some $0 \leq i \leq M-1$. Since $h$ is monotonically increasing by (i) and (iii) above, we have
$$h^{i+1}(n)=h(h^i(n))\geq h(n_i) \geq n_{i+1},$$
by using the induction hypothesis and the assumption on the numbers $n_i$. This completes the proof.
\end{proof}

\begin{proof}[Proof of Theorem \ref{thm_linearly_embed_upper_bound}]
Let $\HH=\HH_0$ be a linearly $d$-embeddable $(d+1)$-uniform hypergraph on $n$ vertices. Starting from the index $i=0$, we conduct the following iterative procedure inductively until we meet an empty hypergraph:
\begin{enumerate}
\item Select an independent set $S_i$ of maximum size in $\HH_i$, and 
\item let $\HH_{i+1}=\HH_i-S_i$, that is, let $\HH_{i+1}$ be the hypergraph obtained from $\HH_i$ by removing $S_i$ from $\HH_i$.
\end{enumerate}
This process terminates after a finite number of steps, say $M$. So, we obtain $\HH_M$ as the final empty hypergraph in this process. Note that $S_0, \dots, S_{M-1}$ are pairwise disjoint independent sets of $\HH$ which cover all vertices of $V(\HH)$. Hence $S_0, \dots, S_{M-1}$ form a proper $M$-coloring of $\HH$.

Let $n_i= \vert V(\HH_i) \vert $ for every $0 \leq i\leq M$. Note that $n_i \geq 1$ for $0 \leq i \leq M-1$ and $n_M=0$. Also, by Theorem \ref{thm_upper_bound_pach} and Lemma \ref{lemma_alteration}, we have
$$\alpha(\HH_i)\geq \beta_{d,d+1}  n_i^{\frac{(d+1)-\lceil d/2\rceil}{d}},$$ which implies $$n_{i+1}\leq n_i - \beta_{d,d+1}  n_i^{\frac{(d+1)-\lceil d/2\rceil}{d}}$$
for every $0 \leq i \leq M-1$. Hence, the sequence $(n_0, n_1, \dots, n_M)$ satisfies all conditions in Lemma \ref{lemma_iteration} with $t= \frac{(d+1)-\lceil d/2\rceil}{d}$ and $\beta= \beta_{d,d+1}$. Therefore, using Lemma \ref{lemma_iteration}, we have
$$\chi(\HH) \leq M \leq  \gamma_{t, \beta} n^{1-t}+1=\gamma_{t, \beta} n^{\frac{\lceil d/2\rceil-1}{d}}+1$$
where $\gamma_{t, \beta}=\frac{1}{\beta (1-t)}$. 
\end{proof}

\begin{remark}
In \cite{coloring_d-embeddable}, the authors obtained the bound $\chi(\HH) \in O(n^{\frac{\lceil d/2\rceil}{d}})$ for a $(d+1)$-uniform hypergraph $\HH$ on $n$ vertices which is linearly $d$-embeddable. Theorem \ref{thm_linearly_embed_upper_bound} is a slight improvement on this upper bound for such hypergraphs.
\end{remark}

\subsection{Generalization of Theorem \ref{thm_main_upper_bound} for simplicial spheres} \label{subsection_upper_bound_embeddable}
\begin{remark}
After the first version of this paper was available, it was first noticed by Andreas Holmsen that Theorem \ref{thm_main_upper_bound} should hold more generally by the hard Lefschetz property of simplicial spheres. This subsection is mainly based on Eran Nevo's accounts commnunicated with us, and also on insightful discussions with Vasiliki Petrotou and Geva Yashfe.
\end{remark}

The same asymptotic upper bound in Theorem \ref{thm_upper_bound_pach} holds for subcomplexes of a simplicial $d$-sphere, where we count only $d$-dimensional faces: In \cite{kalai_generalized_upper_bound_theorem_embeddable_complexes}, Kalai already obtained such an upper bound theorem for subcomplexes of the boundary complex of a polytope. One of the ingredients in the proof is the hard Lefschetz property for polytopes which Stanely proved in \cite{stanley_hard_Lefschetz}. By using the same argument by Kalai with the recent breakthroughs of the hard Lefschetz property of simplicial spheres \cite{g-conjecture_Karim,g-conjecture_Vasso-Papa, g-conjecture_Karim-Vasso-Papa}, we obtain the desired upper bound theorem for subcomplexes of a simplicial $d$-sphere. By using the same argument in Section \ref{section_upper_bound} with this upper bound theorem, we conclude the following.

\begin{thm}\label{thm_main_upper_bound_simplicial_sphere}
Let $\sK$ be a simplicial $d$-sphere on $n$ vertices. Then, the facet hypergraph $\FF(\sK)$ has the chromatic number 
$$\chi(\FF(\sK))\in O(n^{\frac{\lceil d/2\rceil-1}{d}}).$$
\end{thm}

\section{Final remarks} \label{section_final_remark}
We close with some additional problems for future study.

\medskip

An obvious open problem we can ask is to improve the  lower and upper bounds for the transversal ratio and the chromatic number of a simplicial sphere regarding Questions \ref{question_trans} and \ref{question_chi}. 
More precisely, for a transversal ratio we are interested in the parameter (which we similarly defined at Question \ref{Question_stacked_sphere} for stacked spheres, and was also introduced in \cite{novik_transversal_open_problem} recently)
$$\mu_d:=\limsup_{n \to \infty} \max\{\tau(\sK)/n: \textrm{$\sK$ is a simplicial $(d-1)$-sphere on $n$ vertices}\}$$
for every $d\geq 4$. In terms of $\mu_d$, results in this paper can be summarized as follows.
\begin{itemize}
    \item $\mu_3=1/2$ by Proposition \ref{prop_basic_3-polytope} and Theorem \ref{thm_plane_tight}.

    \item $\mu_4 \geq 11/21$ by Theorem \ref{thm_trans}.

    \item For $d\geq 5$, when $d$ is even $\mu_d \geq 1/2$ by Proposition \ref{prop_basic_cyclic_col} (which is originally from \cite{holmsen_surrounding}), and
 when $d$ is odd $\mu_d \geq 2/(d+1)$ by Theorem \ref{thm_plane_tight}.
\end{itemize}
These lower bounds might not be tight for $d\geq 4$. More seriously, we do not have any non-trivial upper bound for $\mu_d$. Unfortunately, Theorems \ref{thm_main_upper_bound} and \ref{thm_main_upper_bound_simplicial_sphere} only give a trivial upper bound $\mu_d\leq 1$. While this trivial upper bound might be tight, it seems that we are far from completely closing the gaps between upper and lower bounds. To understand behaviour of $\mu_d$, we need to understand more about incidences between facets and vertices of a simplicial sphere, and how they differ from the incidence structure of a general hypergraph.

\begin{remark}
  In \cite{novik_transversal_open_problem}, Novik and Zheng recently gave two constructions which shows $\mu_d \geq 2/5$ for every odd $d\geq 5$. This, together with Proposition \ref{prop_basic_cyclic_col}, gives a constant positive lower bound on $\mu_d$ for all dimensions $d \geq 5$.
\end{remark}

\medskip 

For colorings, recall that we found infinitely many simplicial 3-spheres which are not 2-colorable. We ask if we can also find non 2-colorable simplicial spheres for higher dimensions $d>3$. 

\begin{Q}\label{question_not-2-colorable_other-dim}
For every dimension $d>3$, is there a non 2-colorable simplicial $d$-sphere?
\end{Q}

For a fixed odd dimension $d>3$, it might be possible to construct a non 2-colorable simplicial $d$-sphere using the similar inductive computational process. However, the process might not easily work for even dimensions, because even dimensional cyclic polytopes have totally different behaviours on transversals from odd dimensional ones. So, to find a non 2-colorable simplicial $d$-sphere for even $d$, we might need a different approach to obtain it.

\medskip

On the other hand, it is natural to ask similar problems on 2-dimensional surfaces beyond spheres.	Albertson and Hutchinson showed \cite{genus_cutting} that given an $n$-vertex triangulation $\sK$ of an orientable surface of genus $g$, which we denote by $S_g$, there exists a cycle in the graph $G(\sK)$ of length $\leq \sqrt{2n}$ which is non-separating and non-contractible. It follows that there is a set $T$ of $\leq g \sqrt{2n}$ vertices whose removal leaves a planar simplicial complex $\sK'$. Since $\sK'$ in turn has a transversal of size at most $ \vert V(\sK') \vert /2$ by Proposition \ref{prop_basic_3-polytope}, we obtain the following.

\begin{prop}\label{prop_higher_genus_transversal}
Any $n$-vertex triangulation $\sK$ of $S_g$ has $\tau(\sK) \leq n/2 +g \sqrt{2n}$.
\end{prop}

Moreover, the connected sum of a minimal triangulation of $S_g$ and the tight spherical example from Theorem \ref{thm_plane_tight} yields that the leading term $n/2$ above is asymptotically tight. However, we do not know if the second term $g\sqrt{2n}$ can be reduced to a constant or not, when $g$ is fixed. So we ask the following question.

\begin{Q}
What is the behaviour of the function
$$\tau_g(n)=\max\{\tau(\sK): \textrm{$\sK$ is an $n$-vertex triangulation of $S_g$}\}$$ as $n \to \infty$ when $g$ is fixed?
\end{Q}

\section*{Acknowledgments}
We would like to thank Andreas Holmsen for introducing Question \ref{question_trans}. We are also grateful to Minho Cho, Chris Eppolito, Jinha Kim, Minki Kim, Eran Nevo, Isabella Novik, Vasiliki Petrotou, Geva Yashfe and Hailun Zheng for insightful discussions. The authors thank the anonymous referees for their helpful comments that improved the quality of the manuscript.

\bibliographystyle{hplain}
\bibliography{bibliography}

\appendix

\newpage
\section{Construction of the neighborly sphere \texorpdfstring{$\sK_{21}$}{}}\label{section_appendix}
	Here we list the computational results on Theorem \ref{thm_construction}. The neighborly sphere $\sK_{21}$ was constructed from the cyclic polytope $C_4(7)$ using single element extensions. We denote the facets of $C_4(7)$ as follows.
	\begin{align*}
	\FF(C_4(7))=&\{\{0, 1, 2, 3\}, \{0, 1, 2, 6\}, \{0, 1, 3, 4\}, \{0, 1, 4, 5\}, \{0, 1, 5, 6\}, \{0, 2, 3, 6\},\\ &\{0, 3, 4, 6\}, \{0, 4, 5, 6\}, \{1, 2, 3, 4\}, \{1, 2, 4, 5\}, \{1, 2, 5, 6\}, \{2, 3, 4, 5\},\\
	&\{2, 3, 5, 6\}, \{3, 4, 5, 6\}\}
	\end{align*}
	
	Since the facet hypergraph $\FF(\sK_{21})$ is quite large, we only list stacked balls which were used to construct $\sK_{21}$ for each step. Note that facets $\Delta_1, \dots, \Delta_m$ in each stacked ball are ordered as described at the beginning of Subsection \ref{subsection_basic_stacked} where a stacked ball is defined.
    
 \begin{align*}
	n=7:\;&\{\{0, 1, 4, 5\}, \{0, 1, 3, 4\}, \{0, 3, 4, 6\}, \{1, 2, 4, 5\}\}\\
	n=8:\;&\{\{2, 3, 5, 6\}, \{0, 2, 3, 6\}, \{1, 2, 5, 6\}, \{0, 3, 6, 7\}, \{0, 4, 6, 7\}\}\\
	n=9:\;&\{\{0, 2, 3, 8\}, \{0, 2, 6, 8\}, \{0, 4, 6, 8\}, \{1, 2, 6, 8\}, \{4, 6, 7, 8\}, \{2, 3, 5, 8\}\}\\
	n=10:\;&\{\{1, 2, 6, 9\}, \{1, 6, 8, 9\}, \{0, 1, 2, 6\}, \{6, 7, 8, 9\}, \{4, 7, 8, 9\}, \{0, 1, 2, 3\},\\ &\{1, 5, 6, 8\}\}\\
	n=11:\;&\{\{1, 2, 4, 7\}, \{1, 3, 4, 7\}, \{1, 2, 5, 7\}, \{0, 1, 5, 7\}, \{3, 4, 6, 7\}, \{4, 6, 7, 9\},\\ &\{4, 7, 9, 10\}, \{1, 2, 5, 8\}\}\\
	n=12:\;&\{\{1, 3, 4, 11\}, \{1, 2, 4, 11\}, \{3, 4, 6, 11\}, \{2, 4, 7, 11\}, \{2, 5, 7, 11\}, \{0, 5, 7, 11\},\\&\{2, 5, 8, 11\}, \{2, 5, 8, 9\}, \{4, 7, 10, 11\}\}\\
	n=13:\;&\{\{1, 2, 8, 9\}, \{1, 2, 9, 10\}, \{1, 2, 3, 10\}, \{1, 2, 3, 4\}, \{0, 1, 3, 10\}, \{2, 8, 9, 12\}, \\&\{2, 6, 9, 10\}, \{0, 1, 3, 7\}, \{2, 8, 11, 12\}, \{5, 8, 9, 12\}\}\\
	n=14:\;&\{\{3, 4, 5, 6\}, \{0, 4, 5, 6\}, \{0, 4, 5, 7\}, \{3, 5, 6, 8\}, \{3, 5, 8, 9\}, \{0, 5, 7, 12\}, \\&\{2, 3, 5, 9\}, \{0, 1, 5, 6\}, \{0, 7, 11, 12\}, \{5, 8, 9, 13\}, \{7, 10, 11, 12\}\}\\
	n=15:\;&\{\{2, 11, 12, 13\}, \{1, 2, 11, 12\}, \{1, 2, 4, 12\}, \{1, 2, 8, 11\}, \{2, 4, 7, 12\},\\& \{1, 3, 11, 12\}, \{4, 7, 10, 12\}, \{2, 4, 5, 7\}, \{7, 10, 12, 14\}, \{2, 9, 12, 13\},\\& \{2, 6, 9, 13\}, \{0, 2, 6, 9\}\}\\
	n=16:\;&\{\{0, 5, 11, 12\}, \{0, 5, 12, 14\}, \{5, 8, 11, 12\}, \{0, 1, 5, 14\}, \{8, 11, 12, 13\},\\&\{2, 8, 11, 13\}, \{2, 11, 13, 15\}, \{5, 7, 12, 14\}, \{4, 5, 7, 14\}, \{2, 6, 13, 15\},\\&\{2, 6, 10, 13\}, \{3, 4, 5, 14\}, \{6, 9, 13, 15\}\}\\
	n=17:\;&\{\{2, 5, 7, 12\}, \{2, 7, 12, 15\}, \{5, 7, 12, 16\}, \{2, 9, 12, 15\}, \{5, 8, 12, 16\},\\& \{9, 12, 13, 15\}, \{7, 12, 14, 15\}, \{0, 2, 9, 15\}, \{0, 2, 6, 15\}, \{5, 8, 11, 16\},\\& \{1, 5, 8, 11\}, \{0, 2, 3, 9\}, \{1, 5, 8, 10\}, \{4, 5, 7, 16\}\}
	\end{align*}
	
	\begin{align*}
	n=18:\;&\{\{0, 6, 15, 17\}, \{2, 6, 15, 17\}, \{0, 6, 9, 15\}, \{2, 7, 15, 17\}, \{2, 5, 7, 17\},\\& \{0, 4, 6, 9\}, \{2, 5, 12, 17\}, \{0, 4, 6, 14\}, \{5, 8, 12, 17\}, \{2, 6, 15, 16\},\\& \{5, 8, 10, 17\}, \{5, 8, 12, 13\}, \{4, 6, 9, 11\}, \{0, 1, 6, 14\}, \{3, 4, 6, 14\}\}\\
	n=19:\;&\{\{6, 9, 10, 13\}, \{6, 7, 9, 10\}, \{6, 10, 13, 16\}, \{6, 7, 9, 11\}, \{2, 6, 10, 16\},\\& \{1, 9, 10, 13\}, \{2, 6, 16, 18\}, \{1, 8, 9, 13\}, \{0, 2, 6, 10\}, \{3, 6, 7, 11\},\\& \{3, 6, 11, 12\}, \{3, 11, 12, 15\}, \{8, 9, 13, 14\}, \{0, 2, 6, 17\}, \{5, 8, 13, 14\},\\& \{4, 6, 11, 12\}\}\\
	n=20:&\{\{11, 12, 13, 15\}, \{11, 12, 15, 19\}, \{4, 11, 12, 19\}, \{3, 12, 15, 19\}, \{3, 6, 12, 19\},\\
	&\{11, 12, 13, 16\}, \{8, 12, 13, 16\}, \{4, 10, 11, 12\}, \{8, 12, 16, 17\}, \{8, 12, 13, 18\},\\
	&\{7, 12, 16, 17\}, \{0, 11, 12, 16\}, \{10, 11, 12, 14\}, \{0, 5, 11, 16\}, \{0, 1, 5, 16\},\\
	&\{2, 8, 13, 16\}, \{4, 9, 10, 11\}\}
	\end{align*}

\medskip

The resulting $\sK_{21}$ is a neighborly sphere on 21 vertices. The hypergraph $\NN_{21}$ is obtained from $\FF(\sK_{21})$ by removing the hyperedge $\{2, 3, 4, 5\}$. $\NN_{21}$ has transversal number 11 and so does $\FF(\sK_{21})$. One example of a minimum transversal of $\FF(\sK_{21})$ is
$$\{0, 1, 2, 3, 4, 5, 9, 10, 16, 17, 20\}.$$

\end{document}